\documentclass[twocolumn]{autart}    %

\usepackage{graphicx}          %

\usepackage{amsmath}
\usepackage{amsfonts}       %
\usepackage{amssymb}

\newcommand{\set}[1]{\left\{#1\right\}}
\newcommand{\norm}[1]{\left\Vert #1 \right\Vert}
\newcommand{\abs}[1]{\left\vert #1 \right\vert}
\newcommand{\ra}{\rightarrow}
\newcommand{\Real}{\mathbb{R}}
\newcommand{\eps}{\varepsilon}
\newcommand{\diag}[1]{\text{diag}\left\{ #1 \right\}}
\newcommand{\E}[1]{\mathbb{E}\left[ #1 \right]}
\newcommand{\F}{\mathcal{F}}

\usepackage{tikz}
\usetikzlibrary{arrows,shapes,automata,petri,positioning,calc,spy}

\tikzset{
    place/.style={
        circle,
        thick,
        draw=black,
        fill=gray!50,
        minimum size=6mm,
    },
        state/.style={
        circle,
        thick,
        draw=blue!75,
        fill=blue!20,
        minimum size=6mm,
    },
}

\newcommand{\BlackBox}{\rule{1.5ex}{1.5ex}}  %
\newenvironment{proof}{\par\noindent{\bf Proof\ }}{\hfill\BlackBox\\[2mm]}

\begin{document}

\begin{frontmatter}

\title{On the Convergence of Reinforcement Learning\\ with Monte Carlo Exploring Starts\thanksref{footnoteinfo}} %

\thanks[footnoteinfo]{This paper was not presented at any IFAC 
meeting.}

\author{Jun Liu}\ead{j.liu@uwaterloo.ca}    %

\address{Department of Applied Mathematics\\
       University of Waterloo\\
       Waterloo, Ontario N2L 3G1, Canada\\}  %


\begin{abstract}                          %
A basic simulation-based reinforcement learning algorithm is the Monte Carlo Exploring States (MCES) method, also known as optimistic policy iteration, in which the value function is approximated by simulated returns and a greedy policy is selected at each iteration. The convergence of this algorithm in the general setting has been an open question. In this paper, we investigate the convergence of this algorithm for the case with undiscounted costs, also known as the stochastic shortest path problem. The results complement existing partial results on this topic and thereby helps further settle the open problem. As a side result, we also provide a proof of a version of the supermartingale convergence theorem commonly used in stochastic approximation. 
\end{abstract}

\end{frontmatter}

\section{Introduction}

Reinforcement learning has gained tremendous popularity in recent years \cite{sutton2018reinforcement}. Simulation-based methods for reinforcement learning or stochastic control have achieved notable success \cite{silver2009reinforcement}. One particularly simple simulation-based method, called Monte Carlo Exploring Starts (MCES), was introduced in detail in the classic book by Sutton and Barto \cite[Chapter 5]{sutton2018reinforcement}. In this method, the value function is estimated by the average simulated returns and the policy is updated using a greedy policy based on the current estimate of the value function. Because of its fundamental simplicity and importance, Sutton and Barto stated that the convergence of the MCES algorithm to the actual optimal value is ``one of the most fundamental open theoretical questions in reinforcement learning"  \cite[p.~99]{sutton2018reinforcement}. 

Partial results on convergence analysis of MCES exist in the literature. Most notably, Tsitsiklis \cite{tsitsiklis2002convergence} proved that MCES, which he termed optimistic policy iteration, converges under two assumptions. First, each state is selected for updating with the same frequency. Second, the problem is strictly discounted with a discount factor less than one. This result was extended to the undiscounted case by Chen \cite{chen2018convergence} under the assumption that all policies are proper (i.e., reaching a terminal state is inevitable under all policies). A more recent result by Wang and Ross \cite{wang2020convergence} proved convergence of MCES under the assumption of optimal policy feed-forward environments, where states cannot be revisited under an optimal policy. We note that the approach taken in \cite{wang2020convergence} mostly uses finite graph and probabilistic argument, whereas the approach in \cite{tsitsiklis2002convergence} (and also \cite{chen2018convergence}) is along the lines of stochastic approximation \cite{bertsekas1996neuro,robbins1951stochastic}. 

In this paper, we investigate the convergence of MCES/optimistic policy iteration in the undiscounted case without the assumption of optimal policy feed-forward environments and without the assumption that all polices are proper. Compared with the results in \cite{tsitsiklis2002convergence,chen2018convergence,wang2020convergence}, we consider both uniform and nonuniform exploring starts. In the uniform case, we extend the results of Tsitsiklis \cite{tsitsiklis2002convergence} to the undiscounted, i.e., stochastic shortest path problem. 
For the case that all policies are proper, our proof differs from that in \cite{chen2018convergence} and is closer in spirit to that of \cite{tsitsiklis2002convergence} (see how Lemma \ref{lem:main} generalizes Lemma 2 in \cite{tsitsiklis2002convergence}). We also discuss how to work around the proper policy assumption. In the nonuniform case, we argue that the choice of stepsize should be component-dependent to agree with the classical version of MCES discussed in \cite{sutton2018reinforcement}. 
We believe the convergence results established here could help further settle the long-standing open problem. As a side result, we also prove a version of the  supermartingale convergence theorem that is commonly used in the literature of stochastic approximation, whose proof, however, is not available in classic books such as \cite{bertsekas1996neuro}. Furthermore, we 
provide an alternative and hopefully more direct treatment of stochastic approximations directly based on the supermartingale martingale convergence theorem (cf. Chapter 4 of the classic book \cite{bertsekas1996neuro}), which may be of independent interest. 

The paper is organized as follows. In Section \ref{sec:pre}, we present the problem formulation and the preliminaries for proving the convergence of MCES/optimistic policy iteration. In Section \ref{sec:main}, we present the convergence proof for the case that all policies are proper. We discuss the case without the proper policy assumption in Section \ref{sec:improper} and the case with nonuniform initial exploring in Section \ref{sec:nonuniform}. A simple illustrative example is presented in Section \ref{sec:ex}. Some concluding remarks are presented in Section \ref{sec:con}. The Appendix includes a self-contained treatment of supermartingale convergence and stochastic approximation results. 

\section{Problem formulation and preliminaries}\label{sec:pre}

\subsection{Markov decision problem}

Let $M=(S,A,P)$ be a \textit{Markov decision process}, where $S=\set{1,\cdots,n}$ is a finite set of states, $A$ is a finite set of 
actions, and $P:\,S\times A\times S\ra[0,1]$ is a transition probability function. For each action $a\in A$, we can represent $P(\cdot,a,\cdot)$ as a matrix $P(a)$ whose entries $P_{ij}(a)$ satisfy 
$$
P_{ij}(a)=P(i,a,j)=\mathbb{P}(s_{t+1}=j|s_t=i,a_t=a), 
$$ 
where $\set{(s_t,a_t)_{t=0}^{t=\infty}}\subseteq S\times A$ is an evolution of the MDP $M$. In words, $P_{ij}(a)$ denotes the probability of having a transition from the state $i$ to the state $j$ under the action $a$. 

A \textit{policy} is a function $\mu:\,S\ra A$. Clearly, the set of all policies is finite. We denote this set by $\Pi$. Given a policy $\mu\in \Pi$, we define the cost-to-go value\footnote{We use a cost function formulation as commonly seen in stochastic control, which is equivalent to a reward function formulation in reinforcement learning, albeit the difference of using minimization in place of maximization for values.} of the policy starting from a state $i$ as
$$
J^\mu (i) = \mathbb{E}\left[ \sum_{t=0}^\infty \alpha^t g(s_t,\mu(s_t)) | s_0=i \right],
$$
where $\set{s_t}_{t=0}^\infty$ is a state evolution under the policy $\mu$, $g:\,S\times A\ra\Real$ is the stage cost, and $\alpha\in[0,1]$ is a discount factor. The optimal cost-to-go value $J^*$ is defined as
$$
J^*(i) = \min_{\mu\in\Pi} J^{\mu}(i). 
$$
Since the set of policies is finite, the optimal value is always attainable by an optimal policy. That is, there exists $\mu^*\in\Pi$ such that $J^{\mu^*} = J^*$. A Makov decision problem often is concerned with finding the optimal value $J^*$ and an optimal policy $\mu^*$ (which may not be unique). 

We will primarily be focusing on the so-called \textit{stochastic shortest path problem} in this paper, i.e. the Markov decision problem above with $\alpha=1$. To make the cost-to-go value well-defined, we assume that there exists a terminal state, denoted by 0, and modify the transition probability function to satisfy $\sum_{j=1}^nP_{ij}(a)\le 1$ and $P_{i0}(a)=1-\sum_{j=1}^nP_{ij}(a)$ for all $i\in S$ and $a\in A$. In addition, the terminal state is assumed to be a trap state in the sense that $P_{00}(a)= 1$ and $P_{0j}(a)= 0$ for all $a\in A$ and $j\in S$. We also assume $g(0,a)=0$ for all $a\in A$ such that $J^\mu(0)=0$ for all $\mu\in\Pi$. Hence we do not need to discuss the value at state $0$. 

\subsection{Dynamic programming operators} 

We define two dynamic programming operators $T_{\mu}:\,\Real^n\ra\Real^n$ and $T:\,\Real^n\ra\Real^n$ as follows. Given $J\in\Real^n$ and $\mu\in\Pi$, let 
\begin{equation}\label{eq:tmu}
T^{\mu}J(i) = g(i,\mu(i)) + \alpha\sum_{j=1}^n P_{ij}(\mu(i))J(j),    
\end{equation}
and
\begin{equation}\label{eq:t}
TJ(i) = \min_{a\in A} \left[g(i,a) + \alpha\sum_{j=1}^n P_{ij}(a)J(j)\right].
\end{equation}
For convenience, we can write (\ref{eq:tmu}) in a vector format as 
$$
T_\mu J = g_\mu  + \alpha P_{\mu}J,
$$
where $g_\mu=[g(1,\mu(1))\;g(2,\mu(2))\;\cdots\;g(n,\mu(n))]^T\in\Real^n$ and $P_{\mu}=(P_{ij}(\mu(i)))\in\Real^{n\times n}$. %
It follows that, for each $J\in\Real^n$, there exists $\mu\in\Pi$ such that 
$$
T J =  T_\mu J. 
$$
Such a policy is called a \textit{greedy policy} corresponding to $J$. 

\subsection{Optimistic policy iteration with Monte Carlo policy evaluation}

Following \cite{tsitsiklis2002convergence}, we can write the main procedure of optimistic policy iteration using Monte Carlo simulations for policy evaluation as 
\begin{equation}\label{eq:J}
J_{t+1} = (1-\gamma_t) J_t + \gamma_t (J^{\mu_t}+w_t),
\end{equation}
where $J_t$ is the current value vector, $\gamma_t$ is a scalar stepsize parameter (time-varying but deterministic), and $J^{\mu_t}$ is the expected cost value of the current policy $\mu_t$. Given the current value $J_t$, a greedy policy $\mu_t$ is chosen according to
\begin{equation}\label{eq:mu}
T_{\mu_t} J_t = T J_t. 
\end{equation}
The noise $w_t$ captures the discrepancy between the expected cost $J^{\mu_t}$ and observed cumulative cost $J^{\mu_t}+w_t$. Let $\mathcal{F}_t$ be the natural filtration generated by the process (\ref{eq:J}). Since the observed cumulative cost gives an unbiased estimate $J^{\mu_t}$, we have
$
\mathbb{E}[w_t|\mathcal{F}_t] = 0. 
$
Furthermore, the variance of $w_t$ (conditioned on $\mathcal{F}_t$) is only a function of the initial state and the current policy $\mu_t$. Because the numbers of states and polices are finite, we also have
$
\mathbb{E}[\norm{w_t}^2|\mathcal{F}_t] \le C,
$
for some constant $C$. 

\subsection{Preliminaries}

We present some technical preliminaries for convergence analysis. We focus on the shortest path problem (i.e. $\alpha=1$).  A policy $\mu\in\Pi$ is said to be proper if the terminal state $0$ is reached with probability 1 from any initial state. 
\begin{assum}\label{as:proper}
All policies in $\Pi$ are proper. 
\end{assum}

\begin{assum}\label{as:step}
The stepsize parameter satisfies 
$
\sum_{t=0}^{\infty} \gamma_t=\infty
$
and
$
\sum_{t=0}^{\infty} \gamma_t^2<\infty. 
$
\end{assum}

Based on Assumption \ref{as:proper}, a well-known result is that the dynamic programming operators $T$ and $T_\mu$ are contractive with respect to a 
weighted maximum norm.

\begin{lem}\cite[Proposition 2.2, p.~23]{bertsekas1996neuro}\label{lem:contraction}
If Assumption \ref{as:proper} holds, then there exists some $\beta\in [0,1)$ and a vector $\theta\in\Real^n$ of positive components such that
$$
\sum_{j=1}^nP_{ij}(a)\theta(j)\le \beta\theta(i),\quad \forall i\in S,\quad \forall a\in A.
$$
In particular, this statement implies that
$$
\norm{T_{\mu}J_1-T_{\mu}J_2}_{\theta} \le \beta \norm{J_1-J_2}_{\theta},\quad \forall \mu\in\Pi, \forall J_1,J_2\in\Real^n,
$$
and
$$
\norm{TJ_1-TJ_2}_{\theta} \le \beta \norm{J_1-J_2}_{\theta},\quad \forall J_1,J_2\in\Real^n,
$$
where the weighted maximum norm $\norm{\cdot}_{\theta}$ is defined by $\norm{J}_\theta=\max_{1\le i\le n}\frac{\abs{J(i)}}{\theta(i)}.$ 
\end{lem}

Let $\theta\in\Real^n$ be a vector of positive components. Define $\Theta=\diag{\theta(1),\theta(2),\cdots,\theta(n)}$. Let $\mathbf{1}\in\Real^n$ be the column vector with all components equal to 1. The above lemma shows that, in matrix form, 
\begin{equation}\label{eq:one}
P_\mu\Theta \mathbf{1} \le \beta\Theta \mathbf{1}, \quad \forall \mu\in \Pi, 
\end{equation}
where the inequality is interpreted component-wise\footnote{In the sequel, all vector inequalities are interpreted component-wise.}. We refer to this as a weighted contractive property for $P_\mu$. 

We also recall the following property on the dynamical programming parameters $T$ and $T_{\mu}$ for a stochastic shortest path problem. 

\begin{lem}\cite[Lemma 2.2, p.~21]{bertsekas1996neuro}
For every scalar $c\ge 0$, $J\in\Real^n$, and $\mu\in\Pi$, we have 
\begin{equation}\label{eq:error}
T (J+c\mathbf{1}) \le T J + c \mathbf{1},\quad T_{\mu} (J+c\mathbf{1}) \le T_{\mu} J + c \mathbf{1},
\end{equation}
where $c$ is any nonnegative scalar. If $c$ is negative, then the inequalities are reversed. 
\end{lem}

We can also prove a slight modification of the above lemma using (\ref{eq:one}). 

\begin{lem}\label{lem:error}
Suppose that Assumption \ref{as:proper} holds. For every scalar $c\ge 0$, $J\in\Real^n$, and $\mu\in\Pi$, we have 
\begin{equation*}
T (J+c\Theta \mathbf{1}) \le T J + \beta c \Theta  \mathbf{1},\quad T_{\mu} (J+c\Theta \mathbf{1}) \le T_{\mu} J + \beta c \Theta \mathbf{1},
\end{equation*}
where $c$ is any nonnegative scalar. If $c$ is negative, then the inequalities are reversed. 
\end{lem}

\begin{proof}
Let $\mu$ be a greedy policy corresponding to $J$, i.e., $T_{\mu} J = TJ$.  By (\ref{eq:one}), we have 
\begin{align*}
T (J+c\Theta \mathbf{1})& \le T_\mu(J+c\Theta \mathbf{1})= g_\mu + P_\mu(J+c\Theta \mathbf{1})\\
&= g_\mu + P_\mu J + P_\mu c\Theta \mathbf{1}  = T_\mu J + P_\mu c\Theta \mathbf{1} \\
& \le  T_\mu J + \beta c \Theta \mathbf{1} =  T J + \beta c \Theta \mathbf{1}. 
\end{align*}
The above also shows $T_{\mu} (J+c\Theta \mathbf{1}) \le T_{\mu} J + \beta c \Theta \mathbf{1}$ for any $\mu\in\Pi$. 
\end{proof}

By the contraction mapping theorem, Lemma \ref{lem:contraction} implies the following convergence result. 

\begin{lem}\cite{denardo1967contraction,bertsekas1996neuro}\label{lem:exist}
If Assumption \ref{as:proper} holds, we have, for every  $J\in\Real^n$ and $\mu\in\Pi$, 
$$
\lim_{t\ra\infty} T^t J = J^*,\quad \lim_{t\ra\infty} T^t_{\mu} J = J^{\mu},
$$
where $J^*$ and $J^\mu$ are the unique fixed points of $T$ and $T_{\mu}$, respectively. 
\end{lem}

The next lemma is a modified version of Lemma 2 in \cite{tsitsiklis2002convergence}. Let $\theta\in\Real^n$ and $\Theta=\text{diag}\{\theta(1),\theta(2),\cdots,\theta(n)\}$ be defined above. For the sequence $\set{J_t}_{t=0}^{\infty} \subset \Real^n$, define 
$$
c_t = T J_t - J_t,\quad \lambda_t=\max(c_t,0),\quad t\ge 0, 
$$
where $\max$ is taken component-wise. Then clearly $\lambda_t$  is a nonnegative vector and $c_t\le \lambda_t$. 

\begin{lem}\label{lem:main}
Suppose that Assumption \ref{as:proper} holds. 
For every $t\ge 0$, we have
\begin{enumerate}
\item $T_{\mu_t}^k J_t\le J_t + \frac{\norm{\Theta^{-1}\lambda_t}_{\infty} \Theta\mathbf{1}}{1-\beta}$, for all $k\ge 1$, 
\item $J^{\mu_t}\le J_t + \frac{\norm{\Theta^{-1}\lambda_t}_{\infty} \Theta\mathbf{1}}{1-\beta}$, 
\item $J^{\mu_t}\le T J_t + \frac{\beta\norm{\Theta^{-1}\lambda_t}_{\infty} \Theta\mathbf{1}}{1-\beta}$. 
\end{enumerate}
\end{lem}

\begin{proof}
Note that $\norm{\Theta^{-1}\lambda_t}_{\infty}$ is the weighted maximum norm of $\lambda_t$ with respect to the vector $\theta$. From (\ref{eq:mu}), we have $T_{\mu_t} J_t = T J_t$. It follows that $T_{\mu_t} J_t = J_t + c_t$. Applying $T_{\mu_t}$ to both sides of this equation gives
\begin{align*}
T_{\mu_t}^2 J_t & = T_{\mu_t}(J_t + c_t)  = g_{\mu_t} + P_{\mu_t}(J_t + c_t) \\
& =  T_{\mu_t} J_t + P_{\mu_t}c_t = J_t + c_t + P_{\mu_t}c_t.     
\end{align*}
By induction, we obtain
\begin{equation}\label{eq:tmuk}
T_{\mu_t}^k J_t= J_t + (I+P_{\mu_t}+P_{\mu_t}^2+\cdots+P_{\mu_t}^{k-1})c_t.     
\end{equation}
We have, for $m\ge 1$, 
\begin{align}
P_{\mu_t}^{m}c_t &= P_{\mu_t}^{m}\Theta \Theta^{-1} c_t \le P_{\mu_t}^{m}\Theta \Theta^{-1} \lambda_t \le P_{\mu_t}^{m} \Theta \norm{\Theta^{-1} \lambda_t}_{\infty} \mathbf{1}\notag\\
& =   \norm{\Theta^{-1} \lambda_t}_{\infty} P_{\mu_t}^{m-1} P_{\mu_t} \Theta \mathbf{1} \le  \norm{\Theta^{-1}\lambda_t }_{\infty}  P_{\mu_t}^{m-1}  \beta \Theta \mathbf{1}\notag \\
& = \beta   \norm{\Theta^{-1}\lambda_t }_{\infty} P_{\mu_t}^{m-1} \Theta \mathbf{1} \notag\\
&\le \beta^m \norm{\Theta^{-1}\lambda_t }_{\infty}  \Theta\mathbf{1},\label{eq:pmut}
\end{align}
where the first two inequalities follow from the fact that elements of $P_{\mu_t}^{m-1} P_{\mu_t} \Theta$ and  $P_{\mu_t}^{m-1} P_{\mu_t} \Theta\Theta^{-1}$  are nonnegative and 
we can bound components of $c_t$ with $\lambda_t$ and $\Theta^{-1} \lambda_t$ with $\norm{\Theta^{-1}\lambda_t}_{\infty}\mathbf{1}$, the third inequality follows from (\ref{eq:one}), the last inequality follows from an inductive argument, and the equations follow from straightforward rearrangements. Part of the above inequality also shows that, for $m=0$, $c_t\le  \norm{\Theta^{-1} \lambda_t}_{\infty} \Theta \mathbf{1}$. Hence by (\ref{eq:tmuk}) we obtain
\begin{align}
T_{\mu_t}^k J_t&\le J_t + (1+\beta+\beta^2+\cdots+\beta^{k-1})\norm{\Theta^{-1}\lambda_t}_{\infty} \Theta\mathbf{1}\notag\\
&\le J_t + \frac{\norm{\Theta^{-1}\lambda_t}_{\infty} \Theta\mathbf{1}}{1-\beta}, \label{eq:tmuk2}
\end{align} 
where in the first inequality we used (\ref{eq:pmut}) and the fact that $c_t\le\lambda_t \norm{\Theta^{-1}\lambda_t}_{\infty} \Theta\mathbf{1}$. We proved item (1). Since $\lim_{k\ra\infty} T_{\mu_t}^k J_t = J^{\mu_t}$, we proved item (2) by letting $k\ra\infty$. Finally, applying $T_{\mu_t}$ to both sides of the inequality in item (1) and using the fact $J^{\mu_t}=T_{\mu_t} J^{\mu_t}$, we obtain 
\begin{equation*}
J^{\mu_t} = T_{\mu_t} J^{\mu_t} \le T_{\mu_t}(J_t + \frac{\norm{\Theta^{-1}\lambda_t}_{\infty} \Theta\mathbf{1}}{1-\beta}).     
\end{equation*}
By Lemma \ref{lem:error} and the fact that $T_{\mu_t} J_t= T J_t$ , we obtained item (3). 
\end{proof}

\section{Convergence analysis for the stochastic shortest path problem with proper policies}
\label{sec:main}

The convergence analysis starts with an asymptotic estimate for $c_t = T J_t - J_t$ and $\lambda_t=\max(c_t,0)$. All convergence and asymptotic estimates for random variables in this section are understood in the sense of probability 1. 

\begin{lem}\cite{tsitsiklis2002convergence}\label{lem:ct}
Under Assumption \ref{as:step}, we have 
$$\limsup_{t\ra\infty} c_t \le 0\text{ and }\lim_{t\ra\infty} \lambda_t =0.$$
\end{lem}
\begin{proof}
The proof for $\displaystyle\limsup_{t\ra\infty} c_t \le 0$ was established in \cite{tsitsiklis2002convergence} for the case $\alpha<1$. The same argument holds for $\alpha=1$. Here is an outline of the proof. Since $T_{\mu_t}J=g_{\mu_t} + P_{\mu_t} J$ for any $J\in\Real^n$, we can verify that 
\begin{align*}
    T J_{t+1} &\le T_{\mu_t} J_{t+1} = T_{\mu_t}((1-\gamma_t)J_t +\gamma_t J^{\mu_t} + \gamma_t w_t) \\
    & = J_{t+1}+(1-\gamma_t)(TJ_t-J_t) + \gamma_t v_t,
\end{align*}
where we need to use the fact that $T J_t = T_{\mu_t}J_t$ and $v_t = P_{\mu_t} w_t-w_t$. By the property on $w_t$, we have $\mathbb{E}[v_t\,|\,\mathcal{F}_t]=0$ and $\mathbb{E}[\norm{v_t}^2\,|\,\mathcal{F}_t]\le C'$ for some constant $C'$. Hence, $c_t$ satisfies
$$
c_{t+1} \le (1-\gamma_t)c_t + \gamma_t v_t. 
$$
Consider another iteration 
$$
V_{t+1} = (1-\gamma_t)V_t + \gamma_t v_t. 
$$
If $V_0=c_0$, then a comparison argument shows that $c_t\le V_t$ for all $t\ge 0$. By a standard supermartingale convergence argument on stochastic iterations \cite[Chapter 4, p. 143]{bertsekas1996neuro} (see also Proposition \ref{prop:sa} and Lemma \ref{lem:v} in the Appendix), one can show that $V_t$ converges to 0 in probability 1. Hence, $\limsup_{t\ra\infty} c_t \le 0.$ Since $\lambda_t=\max(c_t,0)$, it follows that $\lim_{t\ra\infty} \lambda_t =0$. 
\end{proof}
In particular, the above lemma shows that, for any $\eps>0$, there exists $t(\eps)$ such that 
$$
\frac{\beta\norm{\Theta^{-1}\lambda_t}_{\infty}\Theta \mathbf{1}}{1-\beta} \le \eps\Theta \mathbf{1},\quad \forall t\ge t(\eps).  
$$
Putting this into Lemma \ref{lem:main}(3) shows that 
$$
J_{\mu_t} \le T J_{t} + \eps \Theta\mathbf{1},\quad \forall t\ge t(\eps). 
$$
By (\ref{eq:J}), we obtain
\begin{align*}
J_{t+1}&= (1-\gamma_t)J_t + \gamma_t (J^{\mu_t}+w_t)\\
& \le (1-\gamma_t)J_t + \gamma_t T J_{t} + \gamma_t \eps \Theta\mathbf{1} + \gamma_t w_t,\quad \forall t\ge t(\eps).     
\end{align*}
Define a mapping $H_\eps:\,\Real^n\ra\Real^n$ as 
$
H_\eps J =  T J + \eps \Theta\mathbf{1}
$
and consider the sequence $\set{Z_t}$ generated by 
\begin{align*}
Z_{t+1} &= (1-\gamma_t)Z_t + \gamma_t T Z_{t} + \gamma_t \eps\Theta \mathbf{1} + \gamma_t w_t\\
&= (1-\gamma_t)Z_t + \gamma_t (H_\eps Z_t  + w_t),\quad t\ge t(\eps),
\end{align*}
and $Z_{t(\eps)}=J_{t(\eps)}$. Then, by comparison, 
\begin{equation}
    \label{eq:JZ}
    J_t\le Z_t,\quad\forall t\ge t(\eps).
\end{equation}

Since $T$ is a contraction under the weighted maximum norm $\norm{\cdot}_\theta$, so is $H_\eps$. By Proposition 4.4 in \cite{bertsekas1996neuro} (see also Proposition \ref{prop:sa} in the Appendix), we know that $Z_t$ converges to the unique fixed point of $H_\eps$, denoted by $Z_\eps^*$. 

The following lemma estimates the fixed point of $H_\eps$ relative to $J^*$. 

\begin{lem}\label{lem:Z}
Under Assumption \ref{as:proper}, we have
$$
J^* -\frac{\eps}{1-\beta}\Theta\mathbf{1}\le Z^*_{\eps}\le J^* +\frac{\eps}{1-\beta}\Theta\mathbf{1}. 
$$
\end{lem}

\begin{proof} By Lemma \ref{lem:error} and $T J^* = J^*$, we have
\begin{align*}
H_\eps(J^*+\frac{\eps}{1-\beta}\Theta\mathbf{1}) &= T(J^*+\frac{\eps}{1-\beta} \Theta\mathbf{1}) + \eps\Theta\mathbf{1} \\
& \le T J^* + \frac{\eps\beta}{1-\beta} \Theta\mathbf{1} + \eps\Theta\mathbf{1}\\
& = J^* +\frac{\eps}{1-\beta}\Theta\mathbf{1}.
\end{align*}
It follows that 
$$
Z^*_{\eps} = \lim_{k\ra\infty}H^k_{\eps}(J^*+\frac{\eps}{1-\beta}\Theta\mathbf{1})\le  J^* +\frac{\eps}{1-\beta}\Theta\mathbf{1},
$$
where we used monotonicity of $H_\eps$ (implied by that of $T$). Similary, by Lemma \ref{lem:error}, we can show that 
\begin{align*}
H_\eps(J^*-\frac{\eps}{1-\beta}\Theta\mathbf{1}) \ge J^* - \frac{\eps}{1-\beta}\Theta\mathbf{1}
\end{align*}
and
$
Z_\eps^*\ge J^*-\frac{\eps}{1-\beta}\Theta\mathbf{1}.
$
\end{proof}
We now state and prove the main result of the paper. 

\begin{thm}\label{prop:main}
Under Assumptions \ref{as:proper} an \ref{as:step}, the sequence $J_t$ generated by the optimistic policy iteration (\ref{eq:J}) and (\ref{eq:mu}), applied to a stochastic shortest path problem, converges to $J^*$, with probability 1.
\end{thm}

\begin{proof}
Given any $\eps>0$, by the argument preceding (\ref{eq:JZ}), there exists $t(\eps)$ such that $J_t\le Z_t$ for all $t\ge t(\eps)$. Since $\displaystyle\lim_{t\ra\infty}Z_t=Z_{\eps}^*$, it follows that $\displaystyle\limsup_{t\ra\infty} J_t\le Z_{\eps}^*$. By Lemma \ref{lem:Z}, we have $\displaystyle\limsup_{t\ra\infty} J_t\le J^* +\frac{\eps}{1-\beta}\Theta\mathbf{1}$. Since the choice of $\eps>0$ is arbitrary, we obtain $\displaystyle\limsup_{t\ra\infty} J_t\le J^*$. By the definition of $J^{\mu_t}$ and $J^*$, we have $J^{\mu_t}\ge J^*$. Hence, (\ref{eq:J}) implies 
$$
J_{t+1} \ge (1-\gamma_t)J_t + \gamma_t J^* + \gamma_t w_t. 
$$
Consider the iteration
$$
Y_{t+1} = (1-\gamma_t)Y_t + \gamma_t J^* + \gamma_t w_t
$$
with $Y_0=J_0$. Then the sequence $\set{Y_{t}}$ converges to $J^*$ (see Proposition 4.4 in \cite{bertsekas1996neuro} or Proposition \ref{prop:sa} in the Appendix). By comparison, $\displaystyle\liminf_{t\ra\infty} J_t\ge J^*$. Hence, $\displaystyle\lim_{t\ra\infty} J_t = J^*$. 
\end{proof}

\section{Relaxing the proper policy assumption}\label{sec:improper}

Assumption \ref{as:proper} requires that all policies are proper. In this section, we discuss how to relax this assumption. For the stochastic shortest path problem, the following relaxed assumption was proposed in \cite{bertsekas1991analysis} (see also \cite[Chapter 2]{bertsekas1996neuro}). 

\begin{assum}\label{as:relax}
There exists at least one proper policy, and every improper policy yields an infinite cost for at least one initial state, i.e., for every improper $\mu\in\Pi$, $\displaystyle J^\mu(i)=\lim_{k\ra\infty}[\sum_{t=0}^{k-1}P_\mu g_\mu]_i=\infty$ for some $i\in S$. 
\end{assum}

\begin{lem}\cite{bertsekas1991analysis}\label{lem:mod}
If Assumption \ref{as:relax} holds, we have, for every  $J\in\Real^n$ and proper $\mu\in\Pi$,
$$
\lim_{t\ra\infty} T^t J = J^*,\quad \lim_{t\ra\infty} T^t_{\mu} J = J^{\mu},
$$
where $J^*$ and $J^\mu$ are the unique fixed points of $T$ and $T_{\mu}$, respectively. 
\end{lem}

There is a problem, however, to analyze the convergence of the optimistic policy iteration (\ref{eq:J}) and (\ref{eq:mu}) under Assumption \ref{as:relax}. Unlike in the standard policy iteration, we cannot guarantee the greedy policy generated by the optimistic iteration is always proper. Hence the value iteration (\ref{eq:J}) will possibly attain infinity and become invalid. To overcome this issue, a natural way would be to let the process terminates with a small probability at each stage. This is equivalent to modifying the Markov decision process $M$, by adding a transition with a small probability to the terminal state under each action, so that it satisfies Assumption \ref{as:proper}. 

A natural question to ask is whether the optimal value of the modified problem stays close to that of the original problem and whether an optimal policy obtained for the modified problem remains an optimal policy for the original problem. 

Formally, we define a modified MDP $\hat{M}=(S,A,\hat{P})$ from the original MDP $M=(S,A,P)$ as follows. Let $\hat{P}_{i0}(a) = P_{i0}(a)+p_\eps\sum_{j=1}^{n}P_{ij}(a)$ and $\hat{P}_{ij}(a)=(1-p_\eps){P}_{ij}(a)$ for all $i,j\in S$ and $a\in A$, where $p_\eps\in (0,1)$ is a small probability to be chosen. Then 
$\hat{P}_{\mu}=(1-p_\eps)P_\mu$ for all $\mu\in\Pi$ and $\hat{M}$ satisfies Assumption \ref{as:proper}. Let $\hat{J}^*$ denote the optimal value for $\hat{M}$ and $\hat{\mu}^*$ a corresponding optimal policy. 

\begin{prop}\label{prop:mod}
Suppose that $M$ satisfies Assumption \ref{as:relax}. For every $\eps>0$, there exists some $\delta>0$ such that, if $p_\eps\in(0,\delta)$, then $\norm{\hat{J}^*-J^*}\le\eps$. Furthermore, if $\delta>0$ is sufficiently small, then $p_\eps\in(0,\delta)$ implies that $\hat{\mu}^*$ is also an optimal policy for $M$. 
\end{prop}

\begin{proof}
The proof consists of two main parts. First we show that, by Assumption \ref{as:relax}, any improper policy for $M$ necessarily has large cost-to-go value for at least one component and hence cannot be optimal for $\hat{M}$ (even if it becomes proper with the modification), provided that $p_\eps$ is chosen sufficiently small. We then show that the value of a proper policy in $\hat{M}$ remains close to its value in $M$, provided that $p_\eps$ is sufficiently small. As a result, the optimal value remains close and optimal policy remains the same for $p_\eps$ chosen sufficiently small. 

Let $\mu$ be an improper policy for $M$. Consider the Jordan normal form of $P_\mu$:
$$
Q^{-1}P_{\mu}Q = \begin{bmatrix}I_p & 0\\0 & C \end{bmatrix},
$$
where $Q$ is a nonsingular matrix and $C$ has spectral radius $\rho(C)<1$. We obtain $I_p$ because the eigenvalue 1 of $P_\mu$ is semisimple \cite[p.~696]{meyer2000matrix}. The dimension of $I_p$ cannot be zero because otherwise $\mu$ would be a proper policy. It follows that 
\begin{equation}\label{eq:sumP}
Q^{-1}\sum_{t=0}^{k-1}P_{\mu}^tQ = \begin{bmatrix} kI_p & 0\\ 0 & \sum_{t=0}^{k-1}C^t
\end{bmatrix},
\end{equation}
where $\sum_{t=0}^\infty C^t = (I-C)^{-1}$. Since $\hat{P}_\mu=(1-p_\eps)P_\mu$, we have
\begin{equation*}
Q^{-1}\sum_{t=0}^{k-1}\hat{P}_{\mu}^t Q = \begin{bmatrix} \frac{1-(1-p_{\eps})^{k}}{p_{\eps}}I_p & 0\\ 0 & \sum_{t=0}^{k-1}(1-p_\eps)^tC^t
\end{bmatrix},
\end{equation*}
where $\sum_{t=0}^\infty (1-p_\eps)^tC^t = (I-(1-p_\eps)C)^{-1}$. 

By Assumption \ref{as:relax}, $\displaystyle\lim_{k\ra\infty}\big[\sum_{t=0}^{k-1}P_\mu^t g_\mu\big]_i=\infty$ for some $i\in S$. This is equivalent to 
$$
\lim_{k\ra\infty}\left[Q \begin{bmatrix} kI_p & 0\\ 0 & \sum_{t=0}^{k-1}C^t
\end{bmatrix} Q^{-1} g_\mu\right]_i=\infty,
$$
which is again equivalent to  
$$
\lim_{k\ra\infty}\left[Q \begin{bmatrix} kI_p & 0\\ 0 & 0
\end{bmatrix} Q^{-1} g_\mu\right]_i=\infty
$$
in view of $\sum_{t=0}^\infty C^t<\infty$. Since $\frac{1-(1-p_{\eps})^{k}}{p_{\eps}}\ra k$, as $p_{\eps}\ra0$, and $\sum_{t=0}^\infty (1-p_\eps)^tC^t = (I-(1-p_\eps)C)^{-1}$ is continuous w.r.t. $p_\eps$ and hence bounded for $p_\eps\in[0,1]$, it is straightforward to verify that, for any $c>0$, there exists $\delta>0$, such that 
\begin{align}
&\lim_{k\ra\infty}\big[\sum_{t=0}^{k-1}\hat{P}_\mu^t g_\mu\big]_i\notag \\
& =\lim_{k\ra\infty}\left[Q \begin{bmatrix} \frac{1-(1-p_{\eps})^{k}}{p_{\eps}}I_p & 0 \\ 0 & \sum_{t=0}^{k-1} (1-p_\eps)^tC^t
\end{bmatrix} Q^{-1} g_\mu\right]_i\notag\\
& >c,\quad \forall p_\eps \in (0,\delta]. \label{eq:c}
\end{align}

Now consider a proper policy $\mu$ for $M$ and let $\Pi_0$ denote the set of all proper policies for $M$. Clearly $\mu$ remains a proper policy for $\hat{M}$. The cost vector for $\mu$ in $\hat{M}$ is the unique solution to 
$$
\hat{J}^{\mu} = \hat{T}_{\mu} \hat{J}^{\mu}. 
$$
Note that $\hat{T}_{\mu}$ changes continuous with respect to $p_\eps$ and $\hat{T}_{\mu}=T_{\mu}$ when $p_\eps=0$. Since $T_{\mu}$ has a unique solution $J^\mu$, it follows that $\hat{J}^{\mu}$ also changes continuously with respect to $p_\eps$. Hence for any $\rho>0$, there exists $\delta>0$ such that 
\begin{equation}\label{eq:rho}
    \norm{\hat{J}^{\mu}-J^{\mu}} < \rho,\quad \forall p_\eps \in (0,\delta],\,\forall \mu\in\Pi_0,
\end{equation}
because the number of policies is finite. Let $J^*$ be the optimal value of $M$ and define
$$
d = \min_{\mu\in\Pi_0\atop J^{\mu}\neq J^*}\norm{J^{\mu}-J^*}. 
$$
Choose any $\rho<\frac{d}{2}$ and $\delta$ accordingly such that (\ref{eq:rho}) holds. Choose 
$
c = \max_{\mu\in \Pi_0} J^{\mu} + \rho
$
and reduce $\delta$ accordingly such that (\ref{eq:c}) holds. In view of (\ref{eq:rho}) and the definition of $c$, any improper policy (w.r.t. $M$) cannot be optimal for $\hat{M}$ for all $p_\eps \in (0,\delta]$. Furthermore, suppose that $\hat{J^*}$ is the optimal value and $\hat{\mu}^*$ is an optimal policy for $\hat{M}$. Then $\hat{\mu}^*$ is proper w.r.t. $M$. We claim that $J^{\hat{\mu}^*}=J^*$ and hence $\hat{\mu}^*$ is an optimal policy for $M$. Suppose this is not the case. Then $J^{\hat{\mu}^*}-J^*\ge d$. Since $\norm{\hat{J}^{\hat{\mu}^*}-J^{\hat{\mu}^*}}<\frac{d}{2}$, it follows that $\hat{J}^{\hat{\mu}^*}>J^*+\frac{d}{2}$. Let ${\mu}^*$ be a proper optimal policy for $M$. Then  (\ref{eq:rho}) implies that $\hat{J}^{\mu^*} <  J^{\mu^*} + \frac{d}{2} =J^* + \frac{d}{2}$. Hence $\hat{J}^{\mu^*}<\hat{J}^{\hat{\mu}^*}$ and $\hat{\mu}^*$ cannot be an optimal policy for $\hat{M}$, which is a contradiction. Thus  $\hat{\mu}^*$ is also an optimal policy for $M$. The proof is complete. 
\end{proof}

\section{The case with nonuniform initial exploration}
\label{sec:nonuniform}

The version of optimistic policy iteration described by (\ref{eq:J}) and (\ref{eq:mu}) is synchronous in the sense that $n$ trajectories are simultaneously observed at each iteration, one for each initial state. It is pointed out in \cite{tsitsiklis2002convergence} that the scenario of picking one single state (randomly, uniformly, and independently) at each iteration to generate a trajectory from and update the cost-to-go value at this state can be captured by the following iteration:
\begin{equation}\label{eq:uniform}
    J_{t+1}(i) = \left\{\begin{aligned}
    &(1-\gamma_t)J_t(i) + \gamma_t (J^{\mu_t}(i) + w_t(i)),  \\
    &\qquad\qquad\text{ with probability } \frac{1}{n}, \\
    & J(i),   \qquad\qquad\qquad\text{otherwise}. 
    \end{aligned}\right.
\end{equation}
Furthermore, this algorithm can be equivalently described in the form
\begin{equation}
    \label{eq:J2}
    J_{t+1} = (1-\frac{\gamma_t}{n})J_t + \frac{\gamma_t}{n}(J^{\mu_t}+v_t),
\end{equation}
where 
$$
v_t(i) = w_t(i) + (n\chi_t(i)-1)(-J_t(i) + J^{\mu_t}(i) + w_t(i)),
$$
where each $\chi_t(i)$ is a random variable satisfying $\chi_t(i)=1$ if state $i$ is selected and $\chi_t(i)=0$ otherwise. Then, it can be shown that $v_t$ satisfies $\mathbb{E}[v_t|\mathcal{F}_t]  = 0$ and 
$
\mathbb{E}[\norm{v_t}^2|\mathcal{F}_t] \le A+B\norm{J_t}^2,
$
for some constants $A$ and $B$, where we used the fact that $J^{\mu_t}$ is bounded, because there are only a finite number of policies. To use an argument similar to that in the proof of Theorem \ref{prop:main}, one needs to show that $\mathbb{E}[\norm{v_t}^2|\mathcal{F}_t]$ is bounded. 
\begin{prop}\label{prop:boundJ}
Under Assumptions \ref{as:proper} and \ref{as:step}, the sequence $J_t$ in (\ref{eq:J2}) is bounded. 
\end{prop}

\begin{proof}
There exists some $D>0$ such that $\norm{J^{\mu_t}}\le D$ for all $t\ge 0$. Boundedness of $J_t$ follows from Proposition 4.7 in \cite{bertsekas1996neuro} (see also Proposition \ref{prop:sa} in the Appendix). 
\end{proof}
Hence, $\mathbb{E}[\norm{v_t}^2|\mathcal{F}_t]$ is bounded, this time by a sequence of random variables $A_t$ that are $\mathcal{F}_t$-adapted and bounded. By a similar argument to the proof of Theorem \ref{prop:main}, one can show that the values $J_t$ generated by (\ref{eq:J2}) converge to $J^*$ under the same assumptions. 

A natural question is whether we can extend (\ref{eq:J2}) to the case where the states are chosen according to a nonuniform distribution such that each state has a non-zero probability of being selected. This would lead to the following update rule for $J_t$:
\begin{equation}
    \label{eq:J3}
    J_{t+1} = (1-\Gamma_t)J_t + \Gamma_t(J^{\mu_t}+v_t),
\end{equation}
where $\Gamma_t=\gamma_t\diag{p(1),p(2),\cdots,p(n)}$ and each $p(i)$ is the probability of state $i$ being selected. It is conjectured in \cite{tsitsiklis2002convergence}  that this may not converge (at least with the proof method therein). We also believe this is the case, although a concrete counterexample is yet to be constructed (see Section \ref{sec:ex} for a numerical example). 

Here we provide a slightly different perspective. We argue that, in the case where the states are selected non-uniformly for updating, the stepsize should be different for each state. In fact, for the classical version of MCES \cite[p.~99, Chapter 5]{sutton2018reinforcement}, we have the stepsize given by 
\begin{equation}\label{eq:step1}
\gamma_i(t)=\frac{1}{n_i(t)}, 
\end{equation}
where $n_i(t)$ is the number of times that state $i$ is selected up to iteration $t$, so that $J_t(i)$ is equal to the average of the simulated cumulative costs for state $i$ up to iteration $t$. This choice of stepsize is non-deterministic, but easy to implement. Alternatively, if we know \textit{a priori}  the probability of selecting each state for updating, we can design a time-varying but deterministic stepsize as 
\begin{equation}\label{eq:step2}
\gamma_t(i)=\frac{\hat{\gamma}_t}{p(i)},     
\end{equation}
where $\hat{\gamma}_t$ is any stepsize satisfying Assumption \ref{as:step}. We prove that both (\ref{eq:step1}) and (\ref{eq:step2}) lead to convergence of the following iteration
\begin{equation}\label{eq:nonuni}
    J_{t+1}(i) = \left\{\begin{aligned}
    &(1-\gamma_t(i))J_t(i) + \gamma_t(i) (J^{\mu_t}(i) + w_t(i)),  \\
    &\qquad\qquad\text{ with probability } p(i), \\
    & J(i),   \qquad\qquad\qquad\text{otherwise}, 
    \end{aligned}\right.
\end{equation}
where $p(i)>0$ is the probability that state $i$ is being selected at each iteration. 

\begin{prop}
Under Assumptions \ref{as:proper} and \ref{as:step}, both choices of stepsizes (\ref{eq:step1}) and (\ref{eq:step2}) lead to convergence of $J_t$  in (\ref{eq:nonuni}) to $J^*$ in probability 1. 
\end{prop}

\begin{proof}
We first consider the case (\ref{eq:step2}). 
We can equivalently write the update rule as 
\begin{equation}
    \label{eq:newJ}
    J_{t+1} = (1-\hat{\gamma}_t)J_t + \hat{\gamma}_t(J^{\mu_t}+\hat{v}_t),
\end{equation}
where $\hat{\gamma}_t$ is any stepsize satisfying Assumption \ref{as:step} and 
$$
\hat{v}_t(i) = w_t(i) + (\frac{\chi_t(i)}{p(i)}-1)(-J_t(i) + J^{\mu_t}(i) + w_t(i)),
$$
where each $\chi_t(i)$ is a random variable satisfying $\chi_t(i)=1$ if state $i$ is selected and $\chi_t(i)=0$ otherwise. One can easily verify that 
$\mathbb{E}[\hat{v}_t|\mathcal{F}_t]  = 0$ and 
$$
\mathbb{E}[\norm{\hat{v}_t}^2|\mathcal{F}_t] \le \hat{A}+\hat{B}\norm{J_t}^2,
$$    
where $\hat{A}$ and $\hat{B}$ are constants. Similar to Proposition \ref{prop:boundJ}, $J_t$ generated by (\ref{eq:newJ}) is bounded. The same argument as in the proof of Theorem \ref{prop:main} can be used to show that $J_t$ converges to $J^*$. 

Now consider (\ref{eq:step1}). We can write the update rule as
\begin{equation}
    \label{eq:newJ2}
    J_{t+1} = (1-\hat{\gamma}_t)J_t + \hat{\gamma}_t(J^{\mu_t}+\hat{v}_t+\hat{u}_t),
\end{equation}
where $\hat{\gamma}_t=\frac{1}{t+1}$, 
\begin{align*}
\hat{v}_t(i) &= w_t(i) + \bigg(\frac{\chi_t(i)}{p(i)}-1\bigg)(-J_t(i) + J^{\mu_t}(i)) \\
& \qquad + \bigg(\frac{(t+1)\chi_t(i)}{n_i(t)}-1\bigg)w_t(i),
\end{align*}
and
$$
\hat{u}_t(i) = \bigg(\frac{t+1}{n_t(i)}-\frac{1}{p(i)}\bigg)\chi_t(i)(-J_t(i)+J^{\mu_t}(i)).
$$
By the strong law of large numbers,  $\frac{n_i(t)}{t+1}\ra p(i)$ in probability 1 as $t\ra\infty$. It is easy to see that there exists  bounded and $\mathcal{F}_t$-adapted sequences $A_t$ and $B_t$ such that $\mathbb{E}[\hat{v}_t|\mathcal{F}_t]  = 0$ and 
$
\mathbb{E}[\norm{\hat{v}_t}^2|\mathcal{F}_t] \le {A}_t+{B}_t\norm{J_t}^2. 
$  
Moreover, there exists an $\mathcal{F}_t$-adapted random 
sequence $\theta_t$ such that $\theta_t\ra 0$ in probability 1, as $t\ra \infty$ and 
$
\norm{\hat{u}_t}\le \theta_t(\norm{J_t}+1).
$
We can then use the same argument as in the proof of Theorem \ref{prop:main} to show $J_t$ converges to $J^*$ in probability 1, in which we need to use Proposition 4.5 in \cite{bertsekas1996neuro} (see also Proposition \ref{prop:sa} in the Appendix). 
\end{proof}

\begin{rem}
All the convergence results obtained for the undiscounted case in this paper can be extended to the case of temporal difference $TD(\lambda)$ and model-free case ($Q$-learning) without much difficulty. Such results are left out due to the space limit. Interested readers should be able to refer to Sections 4 and 5 in \cite{tsitsiklis2002convergence} and combine the argument there with those in Sections \ref{sec:main}--\ref{sec:nonuniform} of this paper.  
\end{rem}

\section{An illustrative example} \label{sec:ex}

We use a simple example (adapted from \cite[Example 5.11]{bertsekas1996neuro}) to illustrate the convergence behaviours of different variants of MCES. 

\begin{exmp}\label{ex1}
We consider a discounted problem
(i.e. $\alpha<1$) with two states and deterministic transitions show in Figure \ref{fig:ex}. Note that a discounted problem can be turned into an equivalent shortest path problem \cite{bertsekas1995dynamic} by adding a terminal state and modifying the transition probability such that each transition has $1-\alpha$ probability reaching the terminal state. 

The states consist of $S=\set{1,2}$ and the actions $A=\set{l,r}$. The transitions from each state under each action can be seen from Figure \ref{fig:ex}. We define the stage cost $g$ as $g(1,r)=g(2,l)=0$ and $g(1,l)=g(2,r)=1$. Intuitively, for each of the two states, the cost to move is 0 and the cost to stay is 1. This example was used in \cite{bertsekas1996neuro} to show possible divergence of iteration (\ref{eq:J3}) when we do not restrict the frequency of selecting each of the two states for value updates.

There are in total four different policies $\mu_l$, $\mu_r$, $\mu_g$, and $\mu_w$, defined as follows: $\mu_l(1)=\mu_l(2)=l$, $\mu_r(1)=\mu_r(2)=r$, $\mu_g(1)=r,\mu_g(2)=l$, and $\mu_w(1)=l,\mu_w(2)=r$. The intuitive meaning of $\mu_l$ is to always move to the node on the left, while $\mu_r$ is to always move to the right. The optimal policy $\mu_g$ moves from each node to the opposite node, and the ``worst" policy $\mu_w$ always stays at the current node. It is straightforward to compute the cost-to-go value for each policy as follows:
\begin{align*}
J^{\mu_w} = \begin{bmatrix} \frac{1}{1-\alpha} \\ \frac{1}{1-\alpha}\end{bmatrix}, J^{\mu_g} = \begin{bmatrix} 0\\0\end{bmatrix},
J^{\mu_l} = \begin{bmatrix} \frac{1}{1-\alpha} \\ \frac{\alpha}{1-\alpha}\end{bmatrix},  
J^{\mu_r} = \begin{bmatrix} \frac{\alpha}{1-\alpha} \\ \frac{1}{1-\alpha}\end{bmatrix}.
\end{align*}
Furthermore, it can be verified that $\mu_l$ is a greedy policy for a value vector $J$, if $J(1)\le J(2)-\frac{1}{\alpha}$, and  $\mu_r$ is a greedy policy for a value vector $J$, if $J(2)\le J(1)-\frac{1}{\alpha}$. The optimal policy $\mu_g$ is a greedy policy for a value vector $J$, if $\abs{J(2)-J(1)}\le \frac{1}{\alpha}$. This is also depicted in Figure \ref{fig:sim}(d), where the two black dashed lines ($J(2)=J(1)\pm\frac{1}{\alpha}$) separate the domains of the different greedy policies.
Note that the optimal policy $\mu_g$ in this example does not satisfy the optimal policy feed-forward environment assumption in \cite{wang2020convergence}, because both states are revisited under the optimal policy. 
\end{exmp}

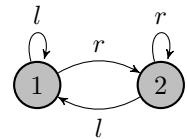
\begin{figure}[ht]
    \centering
\begin{tikzpicture}[node distance=2cm and 1cm,>=stealth',auto, every place/.style={draw}]
    \node [place] (S1) {1};

    \coordinate[node distance=1.1cm,left of=S1] (left-S1);
    \coordinate[node distance=1.1cm,right of=S1] (right-S1);
    
    \node [place] (S2) [right=of S1] {2};

    \path[->] (S1) edge [loop above] node {$l$} ();
    \path[->] (S1) edge [bend left] node {$r$} (S2);
    \path[->] (S2) edge [bend left] node {$l$} (S1);
    \path[->] (S2) edge [loop above] node {$r$} ();
      
\end{tikzpicture}
    \caption{A two-state deterministic system. }%
  \label{fig:ex}
\end{figure}

\begin{figure}[ht]
    \includegraphics[width=0.5\textwidth]{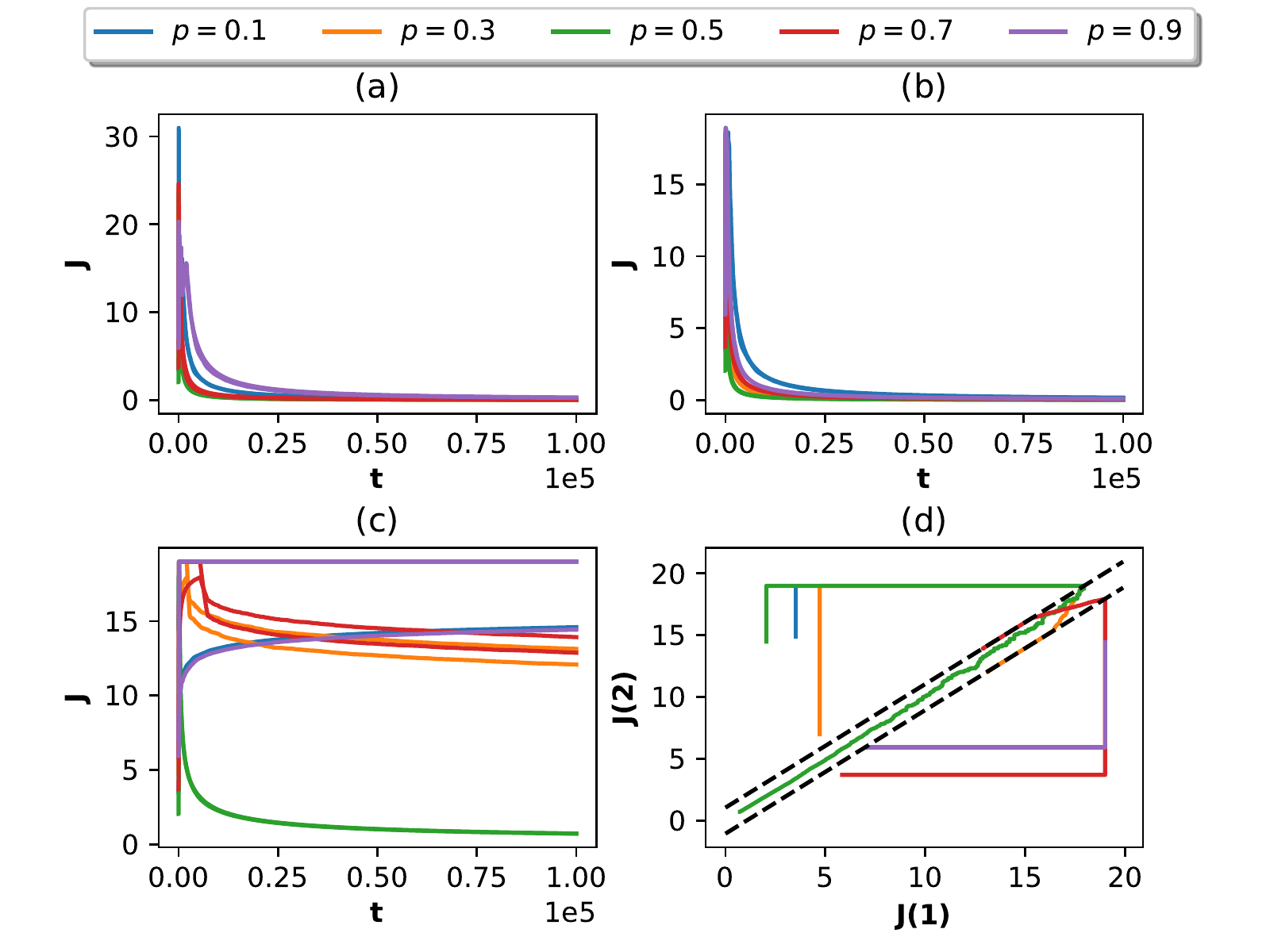}
    \caption{Optimistic policy iteration on Example \ref{ex1} under different choices of stepsize: (a) Iteration (\ref{eq:nonuni} with stepsize choice (\ref{eq:step2}) with $\hat{\gamma}_t=\frac{1}{t+1}$. (b) Iteration (\ref{eq:nonuni} with stepsize choice (\ref{eq:step1}). (c)-(d) Iteration (\ref{eq:J3} with stepsize choice (\ref{eq:step2}) with $\gamma_t=\frac{1}{t+1}$. The black dashed lines in (d) indicate separated domains of the three greedy policies $\mu_l$, $\mu_g$, and $\mu_r$ from top left to bottom right.}
        \label{fig:sim}
\end{figure}

We simulate the optimistic policy iteration (\ref{eq:nonuni}) with different probabilities $p(1)=p$ and $p(2)=1-p$. Figures \ref{fig:sim}(a) and \ref{fig:sim}(b) show convergence of (\ref{eq:nonuni}) using stepsize choices (\ref{eq:step2}) and (\ref{eq:step1}), respectively. Results for iteration (\ref{eq:J3}) are shown in Figure \ref{fig:sim}(c) and \ref{fig:sim}(d). Convergence is observed for (\ref{eq:J3}) only when $p=0.5$ (i.e., in the case of uniform selection).

\section{Conclusions}\label{sec:con}

We investigated the convergence of optimistic policy iteration, also known as Monte Carlo Exploring Starts (MCES), for the stochastic shortest path problem. These results complement known partial results on this topic and thereby help settle this long-standing open question. 

There are at least two possible extensions of this work. First, the results in this paper assume that only the initial state of a simulated trajectory is picked for value updates. It would be interesting to prove convergence for the first-visit and every-visit versions of MCES \cite{sutton2018reinforcement}, in which the first and every state visited on the trajectory, respectively, will be selected for value updating. Second, as pointed out in \cite{tsitsiklis2002convergence}, it would be interesting, and perhaps very challenging, to generalize the results to situations where function approximations are used to  represent values.  

\begin{ack}                               %
The author gratefully acknowledges the NSERC DG, CRC, and Ontario ERA programs for funding support. 
\end{ack}

\bibliographystyle{plain}        %
\bibliography{autosam}           %

\appendix

\section{Martingale Convergence Theorem}    %

The following version of supermartingale convergence theorem stated in \cite{bertsekas1996neuro}, without a proof, is widely used in convergence analysis of stochastic approximation. For completeness, we provide a self-contained proof. %

\begin{thm}\cite[Proposition 4.2]{bertsekas1996neuro}\label{thm:sct}
Let $\set{X_t}$, $\set{Y_t}$, and $\set{Z_t}$ be three sequences of random variables that are adapted to a filtration $\set{\mathcal{F}_t}$. Suppose that the following conditions hold:
\begin{enumerate}
    \item $X_t$, $Y_t$, and $Z_t$ are nonnegative for all $t\ge 0$. 
    \item $\mathbb{E}[Y_{t+1}\,\vert\, \F_t]\le Y_t - X_t + Z_t$. 
    \item $\sum_{t=0}^{\infty}{Z_t}<\infty$ holds in probability 1. 
\end{enumerate}
Then $\sum_{t=0}^{\infty}{X_t}<\infty$ holds in probability 1 and $Y_t$ converges in probability 1 to a nonnegative random variable $Y_\infty$. 
\end{thm}

The book \cite{bertsekas1996neuro} cited \cite{ash1972real} and \cite{neveu1975discrete} for this result. However, the references \cite{ash1972real,neveu1975discrete} do not seem to contain an exact statement of this result, nor a proof. Here we provide a proof of this result for completeness, based on a standard version of the supermartingale convergence theorem below. 

\begin{thm}\label{thm:sct0}
Let $\set{Y_t}$ be a supermartingale bounded in $L^1$, i.e. $\sup_{t\ge 0} \E{\abs{Y_t}}<\infty$. Then $Y_t$ converges in probability 1 to a random variable $Y_\infty$ and $Y_\infty\in L^1$. 
\end{thm}

A proof of this result can be found, e.g., in \cite[p.~109]{williams1991probability}. A variant of  Theorem \ref{thm:sct0} can be proved immediately. 

\begin{cor}\label{cor:sct}
Let $\set{Y_t}$ be a supermartingale. If $\sup_{t\ge 0} \E{Y_t^-}<\infty$, where $Y_t^-$ is the negative part of $Y_t$ defined by $Y_t=\max(0,-Y_t)$. Then $Y_t$ converges in probability 1 to a random variable $Y_\infty\in L^1$. 
\end{cor}

\begin{proof}
Write $\abs{Y_t}=Y_t+2Y_{t}^-$. Since $\set{Y_t}$ is a supermartingale, $\E{Y_t}\le \E{Y_0}$ for all $t\ge 0$. Hence,  $\sup_{t\ge 0} \E{Y_t^-}<\infty$ implies $\sup_{t\ge 0} \E{\abs{Y_t}}<\infty$. The conclusion follows from Theorem \ref{thm:sct0}. 
\end{proof}
Clearly, if $\set{Y_t}$ is a nonnegative supermartingale, then $Y_t^-\equiv 0$ and $Y_t$ converges in probability 1 according to Corollary \ref{cor:sct}. 

To prove Theorem \ref{thm:sct} based on Theorem \ref{thm:sct0}, we also need the following lemma, which says that a stopped supermartingale is still a supermartingale. 

\begin{lem}\label{lem:stopped}
Let $\set{Y_t}$ be a supermartingale and $T$ be a stopping time. Then the stopped process $X^T:= X_{T\wedge t}$, $t=0,1,2,\cdots$, is still a supermartingale. 
\end{lem}

A statement and proof of this result can be found, e.g., in \cite[p.~99]{williams1991probability} or \cite[p.~32]{neveu1975discrete}. 

\textbf{Proof of Theorem \ref{thm:sct}}

For each $t$, define 
$$
W_t = Y_t +  \sum_{s=0}^{t-1} X_s  - \sum_{s=0}^{t-1} Z_s. 
$$
It is straightforward to verify by condition (2) of Theorem \ref{thm:sct} that
\begin{align*}
 \E{W_{t+1}\,\vert\,\mathcal{F}_t} &= \E{Y_{t+1} +  \sum_{s=0}^{t} X_s  - \sum_{s=0}^{t} Z_s\,\vert\,\mathcal{F}_t}\\
 & = \E{Y_{t+1}\,\vert\,\mathcal{F}_t} +  \sum_{s=0}^{t} X_s  - \sum_{s=0}^{t} Z_s \\
 & \le Y_t - X_t + Z_t +  \sum_{s=0}^{t} X_s  - \sum_{s=0}^{t} Z_s\\
 & = Y_t +  \sum_{s=0}^{t-1} X_s  - \sum_{s=0}^{t-1} Z_s = W_t. 
\end{align*}
Hence, $\set{W_t}$ is a supermartingale. 
For each $k\ge 0$, define a stopping time $T_k$ by
$$
T_k = \inf\set{t\ge 0:\,\sum_{s=0}^{t} Z_s \ge k}. 
$$
Then the stopped process $W_{T_k\wedge t}$, $t=0,1,2,\cdots$, according to Lemma \ref{lem:stopped}, is also a supermartingale. Furthermore, by the definition of $T_k$, we have $\sum_{s=0}^{T_k\wedge t-1} Z_s \le k$, %
which implies $W_{T_k\wedge t}\ge -\sum_{s=0}^{T_k\wedge t-1} Z_s \ge -k$. Hence $k+W_{T_k\wedge t}$, $t=0,1,2,\cdots$, is nonnegative supermartingale for each $k\ge 0$. By Corollary \ref{cor:sct}, $\lim_{t\ra\infty}(k+W_{T_k\wedge t})$ exists in probability 1 for each $k\ge 0$. 

Consider the event 
$$
\Omega_{W}^{k}=\set{\omega\in\Omega:\,\lim_{t\ra\infty}(k+W_{T_k\wedge t})\text{ exists}}.
$$
Then %
$P(\Omega_W^k)=1$ for all $k\ge 0$. Let $\Omega_W=\cap_{k=0}^{\infty}\Omega_W^k$. By continuity of probability, we have $P(\Omega_W)=1$. 
Consider also the event 
$$
\Omega_Z=\set{\omega\in\Omega:\,\sum_{t=0}^{\infty}Z_t(\omega)<\infty}.
$$
Then $P(\Omega_Z)=1$. It follows that $P(\Omega_Z\cap\Omega_W)=1$. 

Consider any $\omega\in \Omega_Z\cap \Omega_W$. Since $\omega\in\Omega_Z$, there exists some $k\ge 0$ such that  $\sum_{t=0}^{\infty}Z_t(\omega)< k.$ Hence, for this $k$, we have $T_k(\omega)=\infty$. Since $\omega\in\Omega_W$, $\lim_{t\ra\infty}W_t(\omega)=\lim_{t\ra\infty}(k+W_{T_k(\omega)\wedge t}((\omega)))$ exists. We have proved that $W_t$ converges in probability 1. By the definition of $W_t$ and the fact that $\sum_{t=0}^\infty Z_t<\infty$ in probability 1, $Y_t+\sum_{s=0}^{t-1} X_s$ converges in probability 1. Since $X_t$ is nonnegative, condition (2) of Theorem also holds with $X_t\equiv 0$. Hence, repeating the argument above with $X_t\equiv 0$ would show that $Y_t$ converges in probability 1. This in turn implies $\sum_{t=0}^{\infty} X_t<\infty$ in probability 1. \hfill $\blacksquare$

\section{Stochastic Approximation}         %
Based on the supermartingale convergence theorem, in this section, we provide a more straightforward proof of the convergence result on stochastic approximation arguments we used in this paper. 

Consider a sequence $\set{J_t}$ generated using the update rule
\begin{equation}
    \label{eq:JF}
    J_{t+1}(i) = (1-\gamma_t(i))J_t(i) + \gamma_t(H_tJ_t (i)+ w_t(i) + u_t(i)), 
\end{equation}
where $\gamma_t$, $H_t$, $w_t$, and $u_t$ satisfy the following. 
\begin{assum}\label{as:JF}
We have
\begin{enumerate}
    \item $\sum_{t=0}^\infty \gamma(i)=\infty$ and $\sum_{t=0}^\infty \gamma^2(i)<\infty$ for all $i$. 
    \item There exists a positive vector $\theta\in\Real^n$, a vector  $J^*\in\Real^n$, and sclars $\beta\in[0,1)$ and $D\ge 0$ such that $\norm{H_t J -J^*}_{\theta}\le \beta\norm{J-J^*}_\theta+D$. We also assume that $H_tJ_t$ is $\mathcal{F}_t$-adapted. 
    \item There exist constants $A$ and $B$ such that 
    \[\E{w_t(i)\,\vert\,\mathcal{F}_t}=0,\quad \E{w_t^2(i)\,\vert\,\mathcal{F}_t}\le A_t+B_t\norm{J_t}^2\] for all $i$ and $t$, where $\norm{\cdot}$ is any norm and $A_t$ and $B_t$ are $\mathcal{F}_t$-adapted and bounded. 
    \item There exists an $\mathcal{F}_t$-adapted random sequence $\theta_t$ such that $\lim_{t\ra\infty}\theta_t=0$ in probability 1 and $\abs{u_t(i)}\le \theta_t(\norm{J_t}+1)$ for all $i$ and $t$, where $\norm{\cdot}$ is any norm. We also assume that $u_t$ is $\mathcal{F}_t$-adapted.
\end{enumerate}
\end{assum}

The following result is essentially Propositions 4.7 and 4.5 in \cite{bertsekas1996neuro} combined together. Here we provide a more direct proof from the supermartingale convergence theorem. 

\begin{prop} \label{prop:sa}
Let $J_t$ be generated by (\ref{eq:JF}). Suppose that Assumption \ref{as:JF} holds. Then
\begin{enumerate}
\item $J_t$ is bounded in probability 1, and 
\item $J_t$ converges to $J^*$ in probability 1 if $D=0$.
\end{enumerate} 
\end{prop}

Since the analysis with the weighted maximum norm $\norm{\cdot}_{\theta}$ is very similar to that of the maximum $\norm{\cdot}_{\infty}$. In the following proofs, we only consider the maximum norm and denote it by $\norm{\cdot}$. 

\begin{lem}\label{lem:v}
Consider 
\[
V_{t+1}(i)=(1-\gamma_t(i))V_t(i)+\gamma_t(i)w_t(i),\quad t\ge t_0\ge 0.
\]
\begin{enumerate}
    \item If (3) of Assumption \ref{as:JF} holds with $B=0$ (or $J_t$ is bounded), then $V_t$ converges to 0 in probability 1. 
    \item Let $G_t$ be a nondecreasing $\mathcal{F}_t$-adapted scalar random sequence such that $G_t\ge \mu\norm{J_t}+\nu$ for all $t\ge t_0$, where $\mu$ and $\nu$ are positive constants. Then $\frac{V_t}{G_t}$ converges to 0 in probability 1. 
\end{enumerate}
\end{lem}

\begin{proof}
We prove (2) first. The proof for (1) is a special case. We have
\begin{align*}
V_{t+1}^2(i) &= (1-\gamma_t(i))^2V_t^2(i) + 2\gamma_t(i)(1-\gamma_t(i))V_t(i)w_t(i) \\
& \qquad\qquad + \gamma_t^2(i)w_t^2(i).
\end{align*}
Since $G_t$ is a nondecreasing sequence, we obtain
\begin{align*}
\frac{V_{t+1}^2(i)}{G^2_{t+1}} & \le \frac{V_{t+1}^2(i)}{G^2_{t}} \\
& =  (1-\gamma_t(i))^2\frac{V_t^2(i)}{G_t^2} + 2\gamma_t(i)(1-\gamma_t(i))\frac{V_t(i)w_t(i)}{G_t^2} \\
& \qquad\qquad + \gamma_t^2(i)\frac{w_t^2(i)}{G_t^2}.
\end{align*}
Taking condition expectation from both sides and noticing that $\E{w_t(i)\,\vert\,\mathcal{F}_t}=0$ and $V_t$ and $G_t$ are adapted to $\mathcal{F}_t$ and independent of $w_t$, we have
\begin{align*}
\E{\frac{V_{t+1}^2(i)}{G^2_{t+1}}\,\bigg\vert\,\mathcal{F}_t} & \le  (1-\gamma_t(i))^2\frac{V_t^2(i)}{G_t^2}  + \gamma_t^2(i)\frac{\E{w_t^2(i)}}{G_t^2}\\
&\le (1-\gamma_t(i))^2\frac{V_t^2(i)}{G_t^2}  + \gamma_t^2(i) \frac{A_t+B_t\norm{J_t}^2}{G_t^2} \\
& \le (1 - 2\gamma_t(i) + \gamma_t^2(i))\frac{V_t^2(i)}{G_t^2} + \gamma_t^2(i) K_t,
\end{align*}
for some $\mathcal{F}_t$-adapted and bounded $K_t$, where we used (3) of Assumption \ref{as:JF} and the fact that $G_t\ge \mu\norm{J_t}+\nu$ for all $t$. 

Since $\gamma_t(i)\ra 0$ as $t\ra\infty$, for $t$ sufficiently large, we have $\gamma_t^2(i)\le \gamma_t(i)$ and 
\begin{align*}
\E{\frac{V_{t+1}^2(i)}{G^2_{t+1}}\,\bigg\vert\,\mathcal{F}_t}  \le \frac{V_t^2(i)}{G_t^2} - \gamma_t(i)\frac{V_t^2(i)}{G_t^2} +  \gamma_t^2(i) K_t. 
\end{align*}
Let $Y_t= \frac{V_t^2(i)}{G_t^2}$, $X_t=\gamma_t(i)\frac{V_t^2(i)}{G_t^2}$, and $Z_t=\gamma_t^2(i) K_t$. Since $K_t$ is bounded, we have $\sum_{t}^\infty\gamma_t^2(i) K_t<\infty$ in probability 1. Then the conditions of Theorem \ref{thm:sct}  are satisfied for $t$ sufficiently large. By Theorem \ref{thm:sct}, $Y_t= \frac{V_t^2(i)}{G_t^2}$ converges in probability 1 and $\sum_{t=0}^\infty \gamma_t(i)\frac{V_t^2(i)}{G_t^2}<\infty$ in probability 1, which in turn implies $Y_t= \frac{V_t^2(i)}{G_t^2}$ converges in probability 1 to 0, because otherwise we would have $\sum_{t=0}^\infty \gamma_t(i)\frac{V_t^2(i)}{G_t^2}=\infty$ since $\sum_{t=0}^\infty \gamma_t(i)=\infty$. To prove (1), note that if $B=0$ (or $J_t$ is bounded), we can set $G_t=1$ and prove convergence of $V_t$ in the same way. 
\end{proof}
The first part of the above lemma is Corollary 4.1 in \cite{bertsekas1996neuro}, for which we provide a more direct proof here. The second part appears to be new.

\begin{lem}
\label{lem:y}
Consider  
\[
Y_{t+1}(i)=(1-\gamma_t(i))Y_t(i)+\gamma_t(i)G_t,\quad t\ge t_0,
\]
where $t_0\ge 0$ and $G_t$ is a positive nondecreasing scalar random sequence. Then $\limsup_{t\ra \infty}     \abs{\frac{Y_{t}(i)}{G_{t}}}\le 1$.
\end{lem}

\begin{proof}
We have
\begin{align*}
    \abs{\frac{Y_{t+1}(i)}{G_{t+1}}} &\le  \abs{1-\gamma_t(i)} \abs{\frac{Y_t(i)}{G_{t+1}}} + \gamma_t(i)\\
    & \le (1-\gamma_t(i)) \abs{\frac{Y_t(i)}{G_{t}}} + \gamma_t(i), 
\end{align*}
for $t$ sufficiently large such that $\gamma_t(i)<1$. Consider the iteration 
$$
Z_{t+1}=(1-\gamma_t(i))Z_t + \gamma_t(i), 
$$
with $Z_{t_0}=\frac{Y_{t_0}}{G_{t_0}}$. It is easy to verify that $Z_t\ra 1$, as $t\ra\infty$ (this can in fact be seen as a special case of Lemma \ref{lem:v} with $V_t=Z_t-1$ and $w_t=0$). By comparison, $\frac{Y_t}{G_t}\le Z_t$ for all $t\ge t_0$. Hence, $\limsup_{t\ra \infty}     \abs{\frac{Y_{t}(i)}{G_{t}}}\le 1$.
\end{proof}
\textbf{Proof of Proposition \ref{prop:sa}}  

Note that all the estimates on random variables in this proof are meant to hold in probability 1. 

Fix an $\eta\in (0,1)$ such that $\beta+2\eta<1$. Since $\theta_t\ra0$, as $t\ra \infty$, for any $\eps>0$, there exists $t_0=t_0(\eps)$ such that $\theta_t<\eps$ for all $t\ge t_0$. Since $\gamma_t(i)\ra 0$, we can also assume $t_0$ is picked sufficiently large such that $\gamma_t(i)<1$. 

By Assumption \ref{as:JF}, we have
\begin{align}
\norm{H_t J_t} + \theta_t (\norm{J_t}+1) &\le\beta\norm{J_t} + D + \eps (\norm{J_t}+1) \notag\\
& \le (\beta+\eps)\sup_{0\le s \le t}\norm{J_s} + (D+\eps)\notag\\
& = G_t,\quad\forall t\ge t_0,\label{eq:Ht}
\end{align}
where $G_t:=(\beta+\eps)\sup_{0\le s \le t}\norm{J_s} + (D+\eps)$. Then $G_t$ satisfies the assumptions in Lemmas \ref{lem:v} and \ref{lem:y}. Now consider
$$
Y_{t+1}(i) = (1-\gamma_t(i)) Y_t(i) + \gamma_t(i) G_t,\quad t\ge t_0, 
$$
and 
\[
V_{t+1}(i)=(1-\gamma_t(i))V_t(i)+\gamma_t(i)w_t(i),\quad t\ge t_0. 
\]
Set $Y_{t_0}=J_{t_0}$ and $V_{t_0}=0$.

\textbf{Claim:} We have $Y_t - V_t\le J_t \le Y_t + V_t$ for all $t\ge t_0$. 

\textbf{Proof of the claim:} We prove it by induction. For $t=t_0$, we have $J_{t_0}=Y_{t_0}+V_{t_0}$. Assume the inequality holds for some $t\ge t_0$. By (\ref{eq:Ht}), we have  
\begin{align*}
    J_{t+1}(i) &\le (1-\gamma_t(i))J_t(i) + \gamma_t(i)(H_t J_t(i) + w_t(i) + u_t(i)) \\
    & \le (1-\gamma_t(i))(Y_t(i) + V_t(i)) + \gamma_t(i) G_t +  \gamma_t(i)w_t(i)\\
   & = Y_{t+1}(i) + V_{t+1}(i). 
\end{align*}
The other half of the inequality similarly holds.  \hfill $\blacksquare$

By Lemma \ref{lem:v}, we have $\frac{V_t(i)}{G_t}\ra 0$, as $t\ra \infty$. By Lemma \ref{lem:y}, $\limsup_{t\ra \infty} \abs{\frac{Y_t(i)}{G_t}}\le 1$. For the same $\eps>0$, there exists $T\ge t_0$ such that $\abs{V_t(i)}\le \eps G_t$ and $\abs{Y_t(i)}\le (1+\eps) G_t$ for all $t\ge T$. Hence the above claim implies that $\norm{J_t}\le (1+2\eps)G_t$ for all $t\ge T$. In view of the definition $G_t$ from (\ref{eq:Ht}), we obtain
\begin{equation}\label{eq:Jbound0}
\norm{J_t} \le (1+2\eps) [(\beta+\eps)\sup_{0\le s \le t}\norm{J_s} + (D+\eps)],
\end{equation}
for all $t\ge T$. Fix $\eps\in (0,1)$ sufficiently small such that $\mu:=(1+2\eps)(\beta+\eps)<1$. It follows that
$$
\sup_{0\le s \le t}\norm{J_s} \le \mu \sup_{0\le s \le t}\norm{J_s} + C,\quad \forall t\ge 0,
$$
where $C= \max((1+2\eps) (D+\eps),\sup_{0\le s \le T}\norm{J_s})$. Hence, we obtain an explicit bound for $J_t$ as 
\begin{equation}\label{eq:Jbound}
\norm{J_t}\le \sup_{0\le s \le t}\norm{J_s} \le \frac{C}{1-\mu},\quad \forall t\ge 0.    
\end{equation}
Note that $C$ is a random variable. This proves item (1). 

We now prove item (2). Let $C_0=\frac{C}{1-\mu}$ and $T_0=T$. Then $C_0$ is $\mathcal{F}_T$-measurable. For any $\eps_0\in (0,\eps)$, define $G_t=(\beta+\eps)C_0+\eps$ for $t\ge T_0$. Then (\ref{eq:Ht}) holds with this $G_t$. By repeating the argument preceding (\ref{eq:Jbound0}), we can show that there exists some $T_1\ge T_0$ such that 
\begin{equation}\label{eq:Jbound2}
\norm{J_t} \le (1+2\eps_0) [(\beta+\eps_0)C_0 + \eps_0)],\quad \forall t\ge T_1. 
\end{equation}
We can pick $\eps_0$ sufficiently small such that 
$$
\norm{J_t} \le (\beta+\eta)C_0,\quad \forall t\ge T_1. 
$$
We can inductively show that there exists a sequence $\set{T_k}$ such that
$$
\norm{J_t} \le (\beta+\eta)^k C_0,\quad \forall t\ge T_k. 
$$
Hence $J_t\ra 0$ as $t\ra \infty$. This proves item (2).  \hfill $\blacksquare$.

\end{document}